\documentclass{amsart}%
\usepackage{amsmath}
\usepackage{amssymb}
\usepackage{amsfonts}%
\usepackage{graphicx}
\providecommand{\U}[1]{\protect \rule{.1in}{.1in}}
\newtheorem{theorem}{Theorem}
\theoremstyle{plain}

\newtheorem{corollary}{Corollary}

\newtheorem{definition}{Definition}

\newtheorem{lemma}{Lemma}

\newtheorem{remark}{Remark}

\numberwithin{equation}{section}
\begin{document}
\title[Adams-Spanne type estimates for parabolic sublinear operators]{Adams-Spanne type estimates for parabolic sublinear operators and their
commutators by with rough kernels on parabolic generalized Morrey spaces }
\author{FER\.{I}T G\"{U}RB\"{U}Z}
\address{ANKARA UNIVERSITY, FACULTY OF SCIENCE, DEPARTMENT OF MATHEMATICS, TANDO\u{G}AN
06100, ANKARA, TURKEY }
\curraddr{}
\email{feritgurbuz84@hotmail.com}
\urladdr{}
\thanks{}
\thanks{}
\thanks{}
\date{}
\dedicatory{ }
\begin{abstract}
The aim of this paper is to give Adams-Spanne type estimates for parabolic
sublinear operators and their commutators by with rough kernels generated by
parabolic fractional integral operators under generic size conditions which
are satisfied by most of the operators in harmonic analysis. Their endpoint
estimates are also disposed.

\end{abstract}
\maketitle

\section{Introduction and main results}

Let $S^{n-1}=\left \{  x\in{\mathbb{R}^{n}:}\text{ }|x|=1\right \}  $ denote the
unit sphere on ${\mathbb{R}^{n}}$ $(n\geq2)$ equipped with the normalized
Lebesgue measure $d\sigma \left(  x^{\prime}\right)  $, where $x^{\prime}$
denotes the unit vector in the direction of $x$ and $\alpha_{n}\geq
\alpha_{n-1}\geq \cdots \geq \alpha_{1}\geq1$ be fixed real numbers.

Note that for each fixed $x=\left(  x_{1},\ldots,x_{n}\right)  \in
{\mathbb{R}^{n}}$, the function%
\[
F\left(  x,\rho \right)  =%
{\displaystyle \sum \limits_{i=1}^{n}}
\frac{x_{i}^{2}}{\rho^{2\alpha_{i}}}%
\]
is a strictly decreasing function of $\rho>0$. Hence, there exists a unique
$\rho=\rho \left(  x\right)  $ such that $F\left(  x,\rho \right)  =1$. It is
clear that for each fixed $x\in{\mathbb{R}^{n}}$, the function $F\left(
x,\rho \right)  $ is a decreasing function in $\rho>0$. Fabes and Rivi\'{e}re
\cite{Fabes and Riviere} showed that $\left(  {\mathbb{R}^{n},}\rho \right)  $
is a metric space which is often called the mixed homogeneity space related to
$\left \{  \alpha_{i}\right \}  _{i=1}^{n}$. For $t>0$, we let $A_{t}$ be the
diagonal $n\times n$ matrix%
\[
A_{t}=diag\left[  t^{\alpha_{1}},\ldots,t^{\alpha_{n}}\right]  =%
\begin{pmatrix}
t^{\alpha_{1}} &  & 0\\
& \ddots & \\
0 &  & t^{\alpha_{n}}%
\end{pmatrix}
.
\]

Let $\rho \in \left(  0,\infty \right)  $ and $0\leq \varphi_{n-1}\leq2\pi$,
$0\leq \varphi_{i}\leq \pi$, $i=1,\ldots,n-2$. For any $x=\left(  x_{1}%
,x_{2},\ldots,x_{n}\right)  \in{\mathbb{R}^{n}}$, set%
\begin{align*}
x_{1}  &  =\rho^{\alpha_{1}}\cos \varphi_{1}\ldots \cos \varphi_{n-2}\cos
\varphi_{n-1},\\
x_{2}  &  =\rho^{\alpha_{2}}\cos \varphi_{1}\ldots \cos \varphi_{n-2}\sin
\varphi_{n-1},\\
&  \vdots \\
x_{n-1}  &  =\rho^{\alpha_{n-1}}\cos \varphi_{1}\sin \varphi_{2},\\
x_{n}  &  =\rho^{\alpha_{n}}\sin \varphi_{1}.
\end{align*}
Thus $dx=\rho^{\alpha-1}J\left(  x^{\prime}\right)  d\rho d\sigma(x^{\prime}%
)$, where $\alpha=%
{\displaystyle \sum \limits_{i=1}^{n}}
\alpha_{i}$, $x^{\prime}\in S^{n-1}$, $J\left(  x^{\prime}\right)  =%
{\displaystyle \sum \limits_{i=1}^{n}}
\alpha_{i}\left(  x_{i}^{\prime}\right)  ^{2}$, $d\sigma$ is the element of
area of $S^{n-1}$ and $\rho^{\alpha-1}J\left(  x^{\prime}\right)  $ is the
Jacobian of the above transform. Obviously, $J\left(  x^{\prime}\right)  \in$
$C^{\infty}\left(  S^{n-1}\right)  $ function and that there exists $M>0$ such
that $1\leq J\left(  x^{\prime}\right)  \leq M$ and $x^{\prime}\in S^{n-1}$.

Let $P$ be a real $n\times n$ matrix, whose all the eigenvalues have positive
real part. Let $A_{t}=t^{P}$ $\left(  t>0\right)  $, and set $\gamma=trP$.
Then, there exists a quasi-distance $\rho$ associated with $P$ such that (see
\cite{Coifman-Weiss})

$\left(  1-1\right)  $ $\rho \left(  A_{t}x\right)  =t\rho \left(  x\right)  $,
$t>0$, for every $x\in{\mathbb{R}^{n}}$,

$\left(  1-2\right)  $ $\rho \left(  0\right)  =0$, $\rho \left(  x-y\right)
=\rho \left(  y-x\right)  \geq0$, and $\rho \left(  x-y\right)  \leq k\left(
\rho \left(  x-z\right)  +\rho \left(  y-z\right)  \right)  $,

$\left(  1-3\right)  $ $dx=\rho^{\gamma-1}d\sigma \left(  w\right)  d\rho$,
where $\rho=\rho \left(  x\right)  $, $w=A_{\rho^{-1}}x$ and $d\sigma \left(
w\right)  $ is a measure on the unit ellipsoid $\left \{  w:\rho \left(
w\right)  =1\right \}  $.

Then, $\left \{  {\mathbb{R}^{n},\rho,dx}\right \}  $ becomes a space of
homogeneous type in the sense of Coifman-Weiss (see \cite{Coifman-Weiss}) and
a homogeneous group in the sense of Folland-Stein (see \cite{Folland-Stein}).

Denote by $E\left(  x,r\right)  $ the ellipsoid with center at $x$ and radius
$r$, more precisely, $E\left(  x,r\right)  =\left \{  y\in{\mathbb{R}^{n}:\rho
}\left(  x-y\right)  <r\right \}  $. For $k>0$, we denote $kE\left(
x,r\right)  =\left \{  y\in{\mathbb{R}^{n}:\rho}\left(  x-y\right)
<kr\right \}  $. Moreover, by the property of $\rho$ and the polar coordinates
transform above, we have%
\[
\left \vert E\left(  x,r\right)  \right \vert =%
{\displaystyle \int \limits_{{\rho}\left(  x-y\right)  <r}}
dy=\upsilon_{\rho}r^{\alpha_{1}+\cdots+\alpha_{n}}=\upsilon_{\rho}r^{\gamma},
\]
where $|E(x,r)|$ stands for the Lebesgue measure of $E(x,r)$ and
$\upsilon_{\rho}$ is the volume of the unit ellipsoid on ${\mathbb{R}^{n}}$.
By $E^{C}(x,r)={\mathbb{R}^{n}}\setminus$ $E\left(  x,r\right)  $, we denote
the complement of $E\left(  x,r\right)  $. If we take $\alpha_{1}%
=\cdots=\alpha_{n}=1$ and $P=I$, then obviously $\rho \left(  x\right)
=\left \vert x\right \vert =\left(
{\displaystyle \sum \limits_{i=1}^{n}}
x_{i}^{2}\right)  ^{\frac{1}{2}} $, $\gamma=n$, $\left(  {\mathbb{R}^{n},\rho
}\right)  =$ $\left(  {\mathbb{R}^{n},}\left \vert \cdot \right \vert \right)  $,
$E_{I}(x,r)=B\left(  x,r\right)  $, $A_{t}=tI$ and $J\left(  x^{\prime
}\right)  \equiv1$. Moreover, in the standard parabolic case $P_{0}%
=diag\left[  1,\ldots,1,2\right]  $ we have%
\[
\rho \left(  x\right)  =\sqrt{\frac{\left \vert x^{\prime}\right \vert ^{2}%
+\sqrt{\left \vert x^{\prime}\right \vert ^{4}+x_{n}^{2}}}{2},}\qquad x=\left(
x^{\prime},x_{n}\right)  .
\]

Note that we deal not exactly with the parabolic metric, but with a general
anisotropic metric $\rho$ of generalized homogeneity, the parabolic metric
being its particular case, but we keep the term parabolic in the title and
text of the paper, the above existing tradition, see for instance
\cite{Calderon and Torchinsky}.

Suppose that $\Omega \left(  x\right)  $ is a real-valued and measurable
function defined on ${\mathbb{R}^{n}}$. Suppose that $S^{n-1}$ is the unit
sphere on ${\mathbb{R}^{n}}$ $(n\geq2)$ equipped with the normalized Lebesgue
surface measure $d\sigma$. Let $\Omega \in L_{s}(S^{n-1})$ with $1<s\leq \infty$
be homogeneous of degree zero with respect to $A_{t}$ ($\Omega \left(
x\right)  $ is $A_{t}$-homogeneous of degree zero), that is, $\Omega
(A_{t}x)=\Omega(x),~$for any$~t>0,$ $x\in{\mathbb{R}^{n}}$. We define
$s^{\prime}=\frac{s}{s-1}$ for any $s>1$.

Note that we deal not exactly with the parabolic metric, but with a general
anisotropic metric $\rho$ of generalized homogeneity, the parabolic metric
being its particular case, but we keep the term parabolic in the title and
text of the paper, the above existing tradition, see for instance
\cite{Calderon and Torchinsky}.

In 1938, Morrey considered regularity of the solution of elliptic partial
differential equations(PDEs) in terms of the solutions themselves and their
derivatives. This is a very famous work(=Morrey spaces $M_{p,\lambda}\left(
{\mathbb{R}^{n}}\right)  $) by Morrey \cite{Morrey}. We define parabolic
Morrey spaces $M_{p,\lambda,P}\left(  {\mathbb{R}^{n}}\right)  $ via the
following norm. Let $f\in L_{p}({\mathbb{R}^{n}})$, $0\leq \lambda<\gamma$ and
$1<p<\infty$. Then define%
\[
\left \Vert f\right \Vert _{M_{p,\lambda,P}\left(  {\mathbb{R}^{n}}\right)
}=\sup \limits_{x\in{\mathbb{R}^{n}}}\sup_{r>0}r^{-\frac{\lambda}{p}}\, \Vert
f\Vert_{L_{p}(E(x,r))}\equiv \sup_{E}r^{-\frac{\lambda}{p}}\, \Vert
f\Vert_{L_{p}(E(x,r))},
\]
where $E=E(x,r)$ stands for any ellipsoid with center at $x$ and radius $r$.
When $\lambda=0$, $M_{p,\lambda,P}\left(  {\mathbb{R}^{n}}\right)  $ coincides
with the parabolic Lebesgue space $L_{p,P}\left(  {\mathbb{R}^{n}}\right)  $.
If $P=I$, then $M_{p,\lambda,I}({\mathbb{R}^{n}})\equiv M_{p,\lambda
}({\mathbb{R}^{n}})$ and $L_{p,I}\left(  {\mathbb{R}^{n}}\right)  \equiv
L_{p}\left(  {\mathbb{R}^{n}}\right)  $ are the classical Morrey and the
Lebesgue spaces, respectively. Later many people studied parabolic Morrey
spaces from a various point of view. For example, G\"{u}rb\"{u}z
\cite{Gurbuz3} has been a criterion of the boundedness of anisotropic maximal
functions on weighted anisotropic Morrey Spaces.

Let $f\in L_{1}^{loc}\left(  {\mathbb{R}^{n}}\right)  $. The parabolic
fractional maximal operator $M_{\alpha}^{P}$ and the parabolic fractional
integral operator $I_{\alpha}^{P}$ (also known as the parabolic Riesz
potential) are defined respectively by%
\[
M_{\alpha}^{P}f(x)=\sup_{t>0}|E(x,t)|^{-1+\frac{\alpha}{\gamma}}%
\int \limits_{E(x,t)}|f(y)|dy,\qquad0\leq \alpha<\gamma
\]%
\[
I_{\alpha}^{P}f\left(  x\right)  =\int \limits_{{\mathbb{R}^{n}}}\frac{f\left(
y\right)  }{{\rho}\left(  x-y\right)  ^{\gamma-\alpha}}dy,\qquad
0<\alpha<\gamma.
\]

Now, we list the results of the Adams type and Spanne type boundedness of the
parabolic fractional integral operator and also give the relation between the
Adams inequality and the Spanne inequality on parabolic Morrey spaces.

Spanne considered the boundedness of the parabolic fractional integral
operator on parabolic Morrey spaces. The following Theorem \ref{teo1} is in
\cite{Peetre}.

\begin{theorem}
\label{teo1}Let $0<\alpha<\gamma$, $1<p<\frac{\gamma}{\alpha}$, $0<\lambda
<\gamma-\alpha p$. Moreover, let $\frac{1}{p}-\frac{1}{q}=\frac{\alpha}%
{\gamma}$ and $\frac{\lambda}{p}=\frac{\mu}{q}$. Then we have%
\[
\left \Vert I_{\alpha}^{P}f\right \Vert _{M_{q,\mu,P}}\leq C\left \Vert
f\right \Vert _{M_{p,\lambda,P}}.
\]

\end{theorem}

Later, Adams \cite{Adams} proved the following Theorem \ref{teo2}.

\begin{theorem}
\label{teo2}Let $0<\alpha<\gamma$, $1<p<\frac{\gamma-\lambda}{\alpha}$,
$0<\lambda<\gamma-\alpha p$ and $\frac{1}{p}-\frac{1}{q}=\frac{\alpha}%
{\gamma-\lambda}$. Then we have%
\[
\left \Vert I_{\alpha}^{P}f\right \Vert _{M_{q,\mu,P}}\leq C\left \Vert
f\right \Vert _{M_{p,\lambda,P}}.
\]

\end{theorem}

\begin{remark}
The indices $q_{1}$, $q_{2}$ and $\mu$ satisfy the following relations:%
\[
\frac{1}{q_{1}}=\frac{1}{p}-\frac{\alpha}{\gamma},\frac{1}{q_{2}}=\frac{1}%
{p}-\frac{\alpha}{\gamma-\lambda},\frac{\mu}{q_{1}}=\frac{\lambda}{p}.
\]
Since $q_{1}<q_{2}$, by H\"{o}lder's inequality we get%
\[
\left \Vert I_{\alpha}^{P}f\right \Vert _{M_{q_{1},\mu,P}}\leq \left \Vert
I_{\alpha}^{P}f\right \Vert _{M_{q_{2},\lambda,P}}.
\]

Thus, Theorem \ref{teo2} is a sharper result than Theorem \ref{teo1}, in other
words, Theorem \ref{teo2} improves Theorem \ref{teo1} when $1<p<\frac
{\gamma-\lambda}{\alpha}$:%
\[
\left \Vert I_{\alpha}^{P}f\right \Vert _{M_{q_{1},\mu,P}}\leq \left \Vert
I_{\alpha}^{P}f\right \Vert _{M_{q_{2},\lambda,P}}\leq C\left \Vert f\right \Vert
_{M_{p,\lambda,P}}.
\]

\end{remark}

Recall that, for $0<\alpha<\gamma$,%

\[
M_{\alpha}^{P}f\left(  x\right)  \leq \nu_{n}^{\frac{\alpha}{\gamma}%
-1}I_{\alpha}^{P}\left(  \left \vert f\right \vert \right)  \left(  x\right)
\]
holds (see \cite{Li-Yang}, Remark 2.1). Hence Theorems \ref{teo1} and
\ref{teo2} also imply boundedness of the parabolic fractional maximal operator
$M_{\alpha}^{P}$, where $\upsilon_{n}$ is the volume of the unit ellipsoid on
${\mathbb{R}^{n}}$.

Now, provided that $0<\alpha<\gamma$ and $f\in L_{1}^{loc}\left(
{\mathbb{R}^{n}}\right)  $, we recall the definitions of the parabolic
fractional integral operator with rough kernel $I_{\Omega,\alpha}^{P}$ and a
related parabolic fractional maximal operator with rough kernel $M_{\Omega
,\alpha}^{P}$ as follows:%

\[
M_{\Omega,\alpha}^{P}f(x)=\sup_{t>0}|E(x,t)|^{-1+\frac{\alpha}{\gamma}}%
\int \limits_{E(x,t)}\left \vert \Omega \left(  x-y\right)  \right \vert |f(y)|dy
\]
and
\[
I_{\Omega,\alpha}^{P}f(x)=\int \limits_{{\mathbb{R}^{n}}}\frac{\Omega
(x-y)}{{\rho}\left(  x-y\right)  ^{\gamma-\alpha}}f(y)dy,
\]
where $\Omega \in L_{s}(S^{n-1})$ with $1<s\leq \infty$ be homogeneous of degree
zero with respect to $A_{t}$.

If $\alpha=0$, then $M_{\Omega,0}^{P}\equiv M_{\Omega}^{P}$ is the parabolic
maximal operator with rough kernel and we also get the parabolic
Calder\'{o}n--Zygmund singular integral operator with rough kernel $T_{\Omega
}^{P}=I_{\Omega,0}^{P}$. It is obvious that when $\Omega \equiv1$,
$M_{1,\alpha}^{P}\equiv M_{\alpha}^{P}$ and $I_{1,\alpha}^{P}\equiv I_{\alpha
}^{P}$ are the parabolic fractional maximal operator and the parabolic
fractional integral operator, respectively. If $P=I$, then $M_{\Omega,\alpha
}^{I}\equiv M_{\Omega,\alpha}$ is the fractional maximal operator with rough
kernel, and $M_{\Omega,0}^{I}\equiv M$ is the Hardy-Littlewood maximal
operator with rough kernel.

On the other hand, in 1965, Calder\'{o}n \cite{Calderon1} introduced the
commutator $\left[  A,B\right]  $ on $%
\mathbb{R}
$ which is defined by
\[
\left[  A,B\right]  f\left(  x\right)  =A\left(  x\right)  Bf\left(  x\right)
-B\left(  Af\right)  \left(  x\right)  ,
\]
where $A\in Lip\left(
\mathbb{R}
\right)  $ and the operator $B:=\frac{d}{dx}\circ H$, $H$ denotes the Hilbert
transform defined by
\[
Hf\left(  x\right)  =p.v.\frac{1}{\pi}%
{\displaystyle \int \limits_{-\infty}^{\infty}}
\frac{f\left(  y\right)  }{x-y}dy.
\]
Note that the commutator $\left[  A,B\right]  $ can be rewritten as $\left[
A,\sqrt{-\Delta}\right]  $, where $\Delta=\frac{d^{2}}{dx^{2}}$ is the
Laplacian operator on $%
\mathbb{R}
$. Thus, the study of the commutator $\left[  A,B\right]  $ plays an important
role in some characterizations of function spaces and so on (see
\cite{Gurbuz2, Shi} for example).

Let $b$ be a locally integrable function on ${\mathbb{R}^{n}}$, then for
$0<\alpha<\gamma$, we define commutators generated by parabolic fractional
maximal and integral operators with rough kernels and $b$ as follows,
respectively.%
\[
M_{\Omega,b,\alpha}^{P}\left(  f\right)  (x)=\sup_{t>0}|E(x,t)|^{-1+\frac
{\alpha}{\gamma}}\int \limits_{E(x,t)}\left \vert b\left(  x\right)  -b\left(
y\right)  \right \vert \left \vert \Omega \left(  x-y\right)  \right \vert
|f(y)|dy,
\]%
\[
\lbrack b,I_{\Omega,\alpha}^{P}]f(x)\equiv b(x)I_{\Omega,\alpha}%
^{P}f(x)-I_{\Omega,\alpha}^{P}(bf)(x)=\int \limits_{{\mathbb{R}^{n}}%
}[b(x)-b(y)]\frac{\Omega(x-y)}{{\rho}\left(  x-y\right)  ^{\gamma-\alpha}%
}f(y)dy.
\]

Now, we introduce the parabolic bounded mean oscillation space $BMO_{P}%
({\mathbb{R}^{n}})$ following the known ideas of defining bounded mean
oscillation space $BMO({\mathbb{R}^{n}})$ (see \cite{Gurbuz1}) as follows:

\begin{definition}
For each $b\in L_{1}^{loc}({\mathbb{R}^{n}})$ we set%
\[
b_{E(x,r)}=\frac{1}{|E(x,r)|}%
{\displaystyle \int \limits_{E(x,r)}}
b(y)dy.
\]
For every $r>0$, we define
\[
\Vert b\Vert_{\ast}=\sup_{x\in{\mathbb{R}^{n}},\text{ }r>0}\frac{1}{|E(x,r)|}%
{\displaystyle \int \limits_{E(x,r)}}
|b(y)-b_{E(x,r)}|dy<\infty,
\]
and we say that $b\in BMO_{P}({\mathbb{R}^{n}})$ if $\Vert b\Vert_{\ast
}<\infty$. We also define
\[
BMO_{P}({\mathbb{R}^{n}})=\{b\in L_{1}^{loc}({\mathbb{R}^{n}}):\Vert
b\Vert_{\ast}<\infty \}.
\]

If one regards two functions whose difference is a constant as one, then the
space $BMO_{P}({\mathbb{R}^{n}})$ is a Banach space with respect to norm
$\Vert \cdot \Vert_{\ast}$.
\end{definition}

In 1961 John and Nirenberg \cite{John-Nirenberg} established the following
deep property of functions from $BMO_{P}$.

\begin{theorem}
\label{Theorem}\cite{John-Nirenberg} (John-Nirenberg inequality) If $b\in
BMO_{P}({\mathbb{R}^{n}})$ and $E(x,r)$ is an ellipsoid, then%
\[
\left \vert \left \{  x\in E\left(  x,r\right)  \,:\,|b(x)-b_{E\left(
x,r\right)  }|>\xi \right \}  \right \vert \leq|E\left(  x,r\right)  |\exp \left(
-\frac{\xi}{C\Vert b\Vert_{\ast}}\right)  ,~~~\xi>0,
\]
where $C$ depends only on the dimension $\gamma$.
\end{theorem}

Theorem \ref{Theorem} implies that following results:

\begin{corollary}
\cite{John-Nirenberg} Let $b\in BMO_{P}({\mathbb{R}^{n}})$. Then, for any
$p>1$,%
\begin{equation}
\Vert b\Vert_{\ast}\thickapprox \sup_{x\in{\mathbb{R}^{n}},\text{ }r>0}\left(
\frac{1}{|E(x,r)|}%
{\displaystyle \int \limits_{E(x,r)}}
|b(y)-b_{E(x,r)}|^{p}dy\right)  ^{\frac{1}{p}}\label{5.1}%
\end{equation}
is valid.
\end{corollary}

\begin{corollary}
\label{Corollary} Let $b\in BMO_{P}({\mathbb{R}^{n}})$. Then there is a
constant $C>0$ such that
\begin{equation}
\left \vert b_{E(x,r)}-b_{E(x,t)}\right \vert \leq C\Vert b\Vert_{\ast}\left(
1+\ln \frac{t}{r}\right)  \text{ }~\text{for}~0<2r<t,\label{5.2}%
\end{equation}
and for any $q>1$, it is easy to see that%
\[
\left \Vert b-\left(  b\right)  _{E}\right \Vert _{L_{q}\left(  E\right)  }\leq
Cr^{\frac{\gamma}{q}}\Vert b\Vert_{\ast}\left(  1+\ln \frac{t}{r}\right)  .
\]
where $C$ is independent of $b$, $x$, $r$ and $t$.
\end{corollary}

On the other hand, a recent trend in the theory of parabolic Morrey spaces is
to generalize the parameter $\lambda$ to a function so that they can include
the endpoint case and some generalized integral operators. In this sense, the
definition of parabolic generalized Morrey spaces is given as follows.

\begin{definition}
\label{Definition1}\cite{Gurbuz2} \textbf{(parabolic generalized Morrey space)
}Let $\varphi(x,r)$ be a positive measurable function on ${\mathbb{R}^{n}%
}\times(0,\infty)$ and $1\leq p<\infty$. Then, the parabolic generalized
Morrey space $M_{p,\varphi,P}\equiv M_{p,\varphi,P}({\mathbb{R}^{n}})$ is
defined by
\[
M_{p,\varphi,P}\equiv M_{p,\varphi,P}({\mathbb{R}^{n}})=\left \{
\begin{array}
[c]{c}%
f\in L_{p}^{loc}({\mathbb{R}^{n}}):\\
\Vert f\Vert_{M_{p,\varphi,P}}=\sup \limits_{x\in{\mathbb{R}^{n}},\text{ }%
r>0}\varphi(x,r)^{-1}\,|E(x,r)|^{-\frac{1}{p}}\, \Vert f\Vert_{L_{p}%
(E(x,r))}<\infty
\end{array}
\right \}  .
\]
Furthermore, the weak parabolic generalized Morrey space $WM_{p,\varphi
,P}\equiv WM_{p,\varphi,P}({\mathbb{R}^{n}})$ is defined by%
\[
WM_{p,\varphi,P}\equiv WM_{p,\varphi,P}({\mathbb{R}^{n}})=\left \{
\begin{array}
[c]{c}%
f\in WL_{p}^{loc}({\mathbb{R}^{n}}):\\
\Vert f\Vert_{WM_{p,\varphi,P}}=\sup \limits_{x\in{\mathbb{R}^{n}},\text{ }%
r>0}\varphi(x,r)^{-1}\,|E(x,r)|^{-\frac{1}{p}}\, \Vert f\Vert_{WL_{p}%
(E(x,r))}<\infty
\end{array}
\right \}  .
\]

\end{definition}

According to this definition, we recover the parabolic Morrey space
$M_{p,\lambda,P}$ and the weak parabolic Morrey space $WM_{p,\lambda,P}$ under
the choice $\varphi(x,r)=r^{\frac{\lambda-\gamma}{p}}$:
\[
M_{p,\lambda,P}=M_{p,\varphi,P}\mid_{\varphi(x,r)=r^{\frac{\lambda-\gamma}{p}%
}},~~~~~~~~WM_{p,\lambda,P}=WM_{p,\varphi,P}\mid_{\varphi(x,r)=r^{\frac
{\lambda-\gamma}{p}}}.
\]

We now make some conventions. Throughout this paper, we use the symbol
$F\lesssim G$ to denote that there exists a positive consant $C$ which is
independent of the essential variables and not necessarily the same one in
each occurrence such that $F\leq CG$. If $F\lesssim G$ and$\ G\lesssim F$ we
then write$\ F\approx G$ and say that $F$ and $G$ are equivalent. For a fixed
$p\in \left[  1,\infty \right)  $, $p^{\prime}$ denotes the dual or conjugate
exponent of $p$, namely, $p^{\prime}=\frac{p}{p-1}$ and we use the convention
$1^{\prime}=\infty$ and $\infty^{\prime}=1$. Moreover, $\left \Vert
\Omega \right \Vert _{L_{s}\left(  S^{n-1}\right)  }:=\left(  \int
\limits_{S^{n-1}}\left \vert \Omega \left(  z^{\prime}\right)  \right \vert
^{s}d\sigma \left(  z^{\prime}\right)  \right)  ^{\frac{1}{s}}$.

G\"{u}rb\"{u}z \cite{Gurbuz2} proved Spanne type inequalities for parabolic
sublinear operators with a rough kernel generated by parabolic fractional
integrals and their commutators on parabolic generalized Morrey spaces under
generic size conditions which are satisfied by most of the operators in
harmonic analysis. His results can be summarized as follows:

\begin{theorem}
\cite{Gurbuz2}\textbf{(Spanne type result)} Let $\Omega \in L_{s}(S^{n-1})$,
$1<s\leq \infty$, be $A_{t}$-homogeneous of degree zero. Let $0<\alpha<\gamma$,
$1\leq p<\frac{\gamma}{\alpha}$, $\frac{1}{q}=\frac{1}{p}-\frac{\alpha}%
{\gamma}$. Let $T_{\Omega,\alpha}^{P}$ be a parabolic sublinear operator
satisfying condition (\ref{e1}) in Theorem \ref{teo6} below, bounded from
$L_{p}({\mathbb{R}^{n}})$ to $L_{q}({\mathbb{R}^{n}})$ for $p>1$, and bounded
from $L_{1}({\mathbb{R}^{n}})$ to $WL_{q}({\mathbb{R}^{n}})$ for $p=1$. Let
also, for $s^{\prime}\leq p$, $p\neq1$, the pair $(\varphi_{1},\varphi_{2})$
satisfies the condition%
\[
\int \limits_{r}^{\infty}\frac{\operatorname*{essinf}\limits_{t<\tau<\infty
}\varphi_{1}(x,\tau)\tau^{\frac{\gamma}{p}}}{t^{\frac{\gamma}{q}+1}}dt\leq C\,
\varphi_{2}(x,r),
\]
and for $q<s$ the pair $(\varphi_{1},\varphi_{2})$ satisfies the condition%
\[
\int \limits_{r}^{\infty}\frac{\operatorname*{essinf}\limits_{t<\tau<\infty
}\varphi_{1}(x,\tau)\tau^{\frac{\gamma}{p}}}{t^{\frac{\gamma}{q}-\frac{\gamma
}{s}+1}}dt\leq C\, \varphi_{2}(x,r)r^{\frac{\gamma}{s}},
\]
where $C$ does not depend on $x$ and $r$.

Then the operator $T_{\Omega,\alpha}^{P}$ is bounded from $M_{p,\varphi_{1}%
,P}$ to $M_{q,\varphi_{2},P}$ for $p>1$ and from $M_{1,\varphi_{1},P}$ to
$WM_{q,\varphi_{2},P}$ for $p=1$. Moreover, we have for $p>1$%
\[
\left \Vert T_{\Omega,\alpha}^{P}f\right \Vert _{M_{q,\varphi_{2},P}}%
\lesssim \left \Vert f\right \Vert _{M_{p,\varphi_{1},P}},
\]
and for $p=1$%
\[
\left \Vert T_{\Omega,\alpha}^{P}f\right \Vert _{WM_{q,\varphi_{2},P}}%
\lesssim \left \Vert f\right \Vert _{M_{1,\varphi_{1},P}}.
\]

\end{theorem}

\begin{theorem}
\cite{Gurbuz2}\textbf{(Spanne type result)} Let $\Omega \in L_{s}(S^{n-1})$,
$1<s\leq \infty$, be $A_{t}$-homogeneous of degree zero. Let $[b,T_{\Omega
,\alpha}^{P}]$ be a parabolic sublinear operator satisfying condition
(\ref{e2}) in Theorem \ref{teo100} below and bounded from $L_{p}%
({\mathbb{R}^{n}}) $ to $L_{q}({\mathbb{R}^{n}})$. Let $1<p<\infty$,
$0<\alpha<\frac{\gamma}{\alpha}$, $\frac{1}{q}=\frac{1}{p}-\frac{\alpha
}{\gamma}$ and $b\in BMO_{P}\left(
\mathbb{R}
^{n}\right)  $. Let also, for $s^{\prime}\leq p$ the pair $(\varphi
_{1},\varphi_{2})$ satisfies the condition%
\[
\int \limits_{r}^{\infty}\left(  1+\ln \frac{t}{r}\right)  \frac
{\operatorname*{essinf}\limits_{t<\tau<\infty}\varphi_{1}(x,\tau)\tau
^{\frac{\gamma}{p}}}{t^{\frac{\gamma}{q}+1}}dt\leq C\, \varphi_{2}(x,r),
\]
and for $q<s$ the pair $(\varphi_{1},\varphi_{2})$ satisfies the condition%
\[
\int \limits_{r}^{\infty}\left(  1+\ln \frac{t}{r}\right)  \frac
{\operatorname*{essinf}\limits_{t<\tau<\infty}\varphi_{1}(x,\tau)\tau
^{\frac{\gamma}{p}}}{t^{\frac{\gamma}{q}-\frac{\gamma}{s}+1}}dt\leq C\,
\varphi_{2}(x,r)r^{\frac{\gamma}{s}},
\]
where $C$ does not depend on $x$ and $r$.

Then the operator $[b,T_{\Omega,\alpha}^{P}]$ is bounded from $M_{p,\varphi
_{1},P}$ to $M_{q,\varphi_{2},P}$. Moreover,%
\[
\left \Vert \lbrack b,T_{\Omega,\alpha}^{P}]f\right \Vert _{M_{q,\varphi_{2},P}%
}\lesssim \left \Vert b\right \Vert _{\ast}\left \Vert f\right \Vert _{M_{p,\varphi
_{1},P}}.
\]

\end{theorem}

Now motivated by the above background, it is natural to ask whether we can
prove Adams type inequalities for parabolic sublinear operators with a rough
kernel generated by parabolic fractional integrals and their commutators on
the parabolic generalized Morrey spaces? That is the purpose of this paper is
to find the answer to this question.

Our results can be stated as follows.

\begin{theorem}
\label{teo6}\textbf{(Adams type result)} Suppose that $\Omega \in L_{s}%
(S^{n-1})$, $1<s\leq \infty$, is $A_{t}$-homogeneous of degree zero. Let $1\leq
s^{\prime}<p<q<\infty$, $0<\alpha<\frac{\gamma}{p}$ and let $\varphi \left(
x,t\right)  $ satisfies the conditions%
\begin{equation}
\sup_{r<t<\infty}t^{-\gamma}\operatorname*{essinf}\limits_{t<\tau<\infty
}\varphi \left(  x,\tau \right)  t^{\gamma}\leq C\varphi \left(  x,r\right)
\label{31*}%
\end{equation}
and%
\begin{equation}
\int \limits_{r}^{\infty}t^{\alpha}\varphi \left(  x,t\right)  ^{\frac{1}{p}%
}\frac{dt}{t}\leq Cr^{-\frac{\alpha p}{q-p}},\label{32*}%
\end{equation}
where $C$ does not depend on $x\in%
\mathbb{R}
^{n}$ and $r>0$. Let also $T_{\Omega,\alpha}^{P}$ be a parabolic sublinear
operator satisfying conditions%
\begin{equation}
|T_{\Omega,\alpha}^{P}f(x)|\lesssim \int \limits_{{\mathbb{R}^{n}}}\frac
{|\Omega(x-y)|}{{\rho}\left(  x-y\right)  ^{\gamma-\alpha}}%
\,|f(y)|\,dy\label{e1}%
\end{equation}
and%
\begin{equation}
\left \vert T_{\Omega,\alpha}^{P}\left(  f\chi_{E\left(  x_{0},r\right)
}\right)  \left(  x\right)  \right \vert \lesssim r^{\alpha}M_{\Omega}%
^{P}f\left(  x\right) \label{33}%
\end{equation}
holds for any ellipsoid $E\left(  x_{0},r\right)  $.

Then the operator $T_{\Omega,\alpha}^{P}$ is bounded from $M_{p,\varphi
^{\frac{1}{p}},P}$ to $M_{q,\varphi^{\frac{1}{q}},P}$ for $p>1$ and from
$M_{1,\varphi^{\frac{1}{p}},P}$ to $WM_{q,\varphi^{\frac{1}{q}},P}$ for $p=1
$. Moreover, we have for $p>1$%
\[
\left \Vert T_{\Omega,\alpha}^{P}f\right \Vert _{M_{q,\varphi^{\frac{1}{q}},P}%
}\lesssim \left \Vert f\right \Vert _{M_{p,\varphi^{\frac{1}{p}},P}},
\]
and for $p=1$%
\[
\left \Vert T_{\Omega,\alpha}^{P}f\right \Vert _{WM_{q,\varphi^{\frac{1}{q}},P}%
}\lesssim \left \Vert f\right \Vert _{M_{1,\varphi^{\frac{1}{p}},P}}.
\]

\end{theorem}

\begin{corollary}
Under the conditions of Theorem \ref{teo6}, the operators $M_{\Omega,\alpha
}^{P}$ and $I_{\Omega,\alpha}^{P}$ are bounded from $M_{p,\varphi^{\frac{1}%
{p}},P}$ to $M_{q,\varphi^{\frac{1}{q}},P}$ for $p>1$ and from $M_{1,\varphi
,P}$ to $WM_{q,\varphi^{\frac{1}{q}},P}$ for $p=1$.
\end{corollary}

In the case of $\varphi \left(  x,r\right)  =r^{\lambda-\gamma}$,
$0<\lambda<\gamma$ from Theorem \ref{teo6} we get the following Adams type
result (see \cite{Adams}) for the parabolic fractional maximal and integral
operators with rough kernels.

\begin{corollary}
Suppose that $\Omega \in L_{s}(S^{n-1})$, $1<s\leq \infty$, is $A_{t}%
$-homogeneous of degree zero. Let $0<\alpha<\gamma$, $1<p<\frac{\gamma}%
{\alpha}$, $s^{\prime}<p<q<\infty$, $0<\lambda<\gamma-\alpha p$ and $\frac
{1}{p}-\frac{1}{q}=\frac{\alpha}{\gamma-\lambda}$. Then the operators
$M_{\Omega,\alpha}^{P}$ and $I_{\Omega,\alpha}^{P}$ are bounded from
$M_{p,\lambda,P}$ to $M_{q,\lambda,P}$.
\end{corollary}

In the case of $\Omega=1$ from Theorem \ref{teo6}, we get

\begin{corollary}
\label{corollary1}Let $1\leq p<\infty$, $0<\alpha<\frac{\gamma}{p}$, $p<q$,
and let also $\varphi \left(  x,t\right)  $ satisfies conditions (\ref{31*})
and (\ref{32*}). Then the operators $M_{\alpha}$ and $\overline{T}_{\alpha} $
are bounded from $M_{p,\varphi^{\frac{1}{p}},P}$ to $M_{q,\varphi^{\frac{1}%
{q}},P}$ for $p>1$ and from $M_{1,\varphi,P}$ to $WM_{q,\varphi^{\frac{1}{q}%
},P}$ for $p=1$.
\end{corollary}

In the case of $\varphi \left(  x,t\right)  =t^{\lambda-\gamma}$,
$0<\lambda<\gamma$ from Corollary \ref{corollary1} we get Theorem \ref{teo2}.

\begin{theorem}
\label{teo100}\textbf{(Adams type result)} Suppose that $\Omega \in
L_{s}(S^{n-1})$, $1<s\leq \infty$, is $A_{t}$-homogeneous of degree zero. Let
$1<s^{\prime}<p<q<\infty$, $0<\alpha<\frac{\gamma}{p}$, $b\in BMO_{P}\left(
{\mathbb{R}^{n}}\right)  $ and let $\varphi \left(  x,t\right)  $ satisfies the
conditions%
\begin{equation}
\sup_{r<t<\infty}t^{-\frac{\gamma}{p}}\left(  1+\ln \frac{t}{r}\right)
^{p}\operatorname*{essinf}\limits_{t<\tau<\infty}\varphi \left(  x,\tau \right)
t^{\frac{\gamma}{p}}\leq C\varphi \left(  x,r\right)  ,\label{67}%
\end{equation}
and%
\begin{equation}
\int \limits_{r}^{\infty}\left(  1+\ln \frac{t}{r}\right)  t^{\alpha}%
\varphi \left(  x,t\right)  ^{\frac{1}{p}}\frac{dt}{t}\leq Cr^{-\frac{\alpha
p}{q-p}},\label{74}%
\end{equation}
where $C$ does not depend on $x\in%
\mathbb{R}
^{n}$ and $r>0$. Let also $[b,T_{\Omega,\alpha}^{P}]$ be a sublinear operator
satisfying conditions%
\begin{equation}
|[b,T_{\Omega,\alpha}^{P}]f(x)|\lesssim%
{\displaystyle \int \limits_{{\mathbb{R}^{n}}}}
|b(x)-b(y)|\, \frac{|\Omega(x-y)|}{{\rho}\left(  x-y\right)  ^{\gamma-\alpha}%
}\,|f(y)|\,dy\label{e2}%
\end{equation}
and%
\begin{equation}
\left \vert \lbrack b,T_{\Omega,\alpha}^{P}]\left(  f\chi_{E\left(
x_{0},r\right)  }\right)  \left(  x\right)  \right \vert \lesssim r^{\alpha
}M_{\Omega,b}^{P}f\left(  x\right) \label{100}%
\end{equation}
holds for any ellipsoid $E\left(  x_{0},r\right)  $.

Then the operator $[b,T_{\Omega,\alpha}^{P}]$ is bounded from $M_{p,\varphi
^{\frac{1}{p}},P}$ to $M_{q,\varphi^{\frac{1}{q}},P}$. Moreover%
\[
\left \Vert \lbrack b,T_{\Omega,\alpha}^{P}]f\right \Vert _{M_{q,\varphi
^{\frac{1}{q}},P}}\lesssim \Vert b\Vert_{\ast}\left \Vert f\right \Vert
_{M_{p,\varphi^{\frac{1}{p}},P}}.
\]

\end{theorem}

\begin{corollary}
Under the conditions of Theorem \ref{teo100}, the operators $M_{\Omega
,b,\alpha}^{P}$ and $[b,I_{\Omega,\alpha}^{P}]$ are bounded from
$M_{p,\varphi^{\frac{1}{p}},P}$ to $M_{q,\varphi^{\frac{1}{q}},P}$.
\end{corollary}

In the case of $\Omega=1$ from Theorem \ref{teo100}, we get

\begin{corollary}
Let $1<p<\infty$, $0<\alpha<\frac{\gamma}{p}$, $p<q$, $b\in BMO_{P}\left(
{\mathbb{R}^{n}}\right)  $and let also $\varphi \left(  x,t\right)  $ satisfies
conditions (\ref{67}) and (\ref{74}). Then the operators $M_{b,\alpha}^{P}$
and $[b,I_{\alpha}^{P}]$ are bounded from $M_{p,\varphi^{\frac{1}{p}},P}$ to
$M_{q,\varphi^{\frac{1}{q}},P}$.
\end{corollary}

\section{Proofs of the main resuls}

\subsection{Proof of Theorem \ref{teo6}}

\begin{proof}
Let $1<p<\infty$, $0<\alpha<\frac{n}{p}$, $p<q$, and $f\in M_{p,\varphi
^{\frac{1}{p}}}$. Set $E=E\left(  x_{0},r\right)  $ for the parabolic ball
(ellipsoid) centered at $x_{0}$ and of radius $r$ and $2E=E\left(
x_{0},2r\right)  $. We represent $f$ as%
\begin{equation}
f=f_{1}+f_{2},\qquad \text{\ }f_{1}\left(  y\right)  =f\left(  y\right)
\chi_{2kE}\left(  y\right)  ,\qquad \text{\ }f_{2}\left(  y\right)  =f\left(
y\right)  \chi_{\left(  2kE\right)  ^{C}}\left(  y\right)  ,\qquad
r>0\label{e39}%
\end{equation}
and have
\[
\left \vert T_{\Omega,\alpha}^{P}f\left(  x\right)  \right \vert \leq \left \vert
T_{\Omega,\alpha}^{P}f_{1}\left(  x\right)  \right \vert +\left \vert
T_{\Omega,\alpha}^{P}f_{2}\left(  x\right)  \right \vert .
\]

For $T_{\Omega,\alpha}^{P}f_{2}\left(  x\right)  $ we have%
\[
\left \vert T_{\Omega,\alpha}^{P}f_{2}\left(  x\right)  \right \vert
\lesssim \int \limits_{\left(  2kE\right)  ^{C}}\frac{\left \vert f\left(
y\right)  \right \vert \left \vert \Omega \left(  x-y\right)  \right \vert }%
{\rho \left(  x-y\right)  ^{\gamma-\alpha}}dy.
\]

By Fubini's theorem, H\"{o}lder's inequality and $\left[  \text{(2.5) in
\cite{Gurbuz2}}\right]  $, we get%
\begin{align}
\int \limits_{\left(  2kE\right)  ^{C}}\frac{\left \vert f\left(  y\right)
\right \vert \left \vert \Omega \left(  x-y\right)  \right \vert }{\rho \left(
x-y\right)  ^{\gamma-\alpha}}dy  & \approx \int \limits_{\left(  2kE\right)
^{C}}\left \vert f\left(  y\right)  \right \vert \left \vert \Omega \left(
x-y\right)  \right \vert \int \limits_{\rho \left(  x-y\right)  }^{\infty}%
\frac{dt}{t^{\gamma+1-\alpha}}dy\nonumber \\
& \approx \int \limits_{2kr}^{\infty}\int \limits_{2kr\leq \rho \left(  x-y\right)
\leq t}\left \vert f\left(  y\right)  \right \vert \left \vert \Omega \left(
x-y\right)  \right \vert dy\frac{dt}{t^{\gamma+1-\alpha}}\nonumber \\
& \lesssim \int \limits_{2kr}^{\infty}\int \limits_{E\left(  x,t\right)
}\left \vert f\left(  y\right)  \right \vert \left \vert \Omega \left(
x-y\right)  \right \vert dy\frac{dt}{t^{\gamma+1-\alpha}}\nonumber \\
& \lesssim \int \limits_{2kr}^{\infty}\left \Vert f\right \Vert _{L_{p}\left(
E\left(  x,t\right)  \right)  }\left \Vert \Omega \left(  x-\cdot \right)
\right \Vert _{L_{s}\left(  E\left(  x,t\right)  \right)  }\left \vert E\left(
x,t\right)  \right \vert ^{1-\frac{1}{p}-\frac{1}{s}}\frac{dt}{t^{\gamma
+1-\alpha}}\nonumber \\
& \lesssim \int \limits_{2kr}^{\infty}\left \Vert f\right \Vert _{L_{p}\left(
E\left(  x,t\right)  \right)  }\left \Vert \Omega \left(  x-\cdot \right)
\right \Vert _{L_{s}\left(  E\left(  x,t\right)  \right)  }\left \vert E\left(
x,t\right)  \right \vert ^{1-\frac{1}{p}-\frac{1}{s}}\frac{dt}{t^{\gamma
+1-\alpha}}\nonumber \\
& \lesssim \int \limits_{2kr}^{\infty}\left \Vert f\right \Vert _{L_{p}\left(
E\left(  x,t\right)  \right)  }\left \Vert \Omega \left(  x-\cdot \right)
\right \Vert _{L_{s}\left(  E\left(  x,t\right)  \right)  }\left \vert E\left(
x,t\right)  \right \vert ^{1-\frac{1}{p}-\frac{1}{s}}\frac{dt}{t^{\gamma
+1-\alpha}}\nonumber \\
& \lesssim \int \limits_{2kr}^{\infty}t^{\alpha-\frac{\gamma}{p}-1}\left \Vert
f\right \Vert _{L_{p}\left(  E\left(  x,t\right)  \right)  }dt.\label{34}%
\end{align}

Then from conditions (\ref{32*}), (\ref{33}) and inequality (\ref{34}) we get%
\begin{align*}
\left \vert T_{\Omega,\alpha}^{P}f\left(  x\right)  \right \vert  & \lesssim
r^{\alpha}M_{\Omega}^{P}f\left(  x\right)  +\int \limits_{2kr}^{\infty
}t^{\alpha-\frac{\gamma}{p}-1}\left \Vert f\right \Vert _{L_{p}\left(  E\left(
x,t\right)  \right)  }dt\\
& \leq r^{\alpha}M_{\Omega}^{P}f\left(  x\right)  +\left \Vert f\right \Vert
_{M_{p,\varphi^{\frac{1}{p}},P}}\int \limits_{2kr}^{\infty}t^{\alpha}%
\varphi \left(  x,t\right)  ^{\frac{1}{p}}\frac{dt}{t}\\
& \lesssim r^{\alpha}M_{\Omega}^{P}f\left(  x\right)  +r^{-\frac{\alpha
p}{q-p}}\left \Vert f\right \Vert _{M_{p,\varphi^{\frac{1}{p}},P}}.
\end{align*}

Hence choosing $r=\left(  \frac{\left \Vert f\right \Vert _{M_{p,\varphi
^{\frac{1}{p}},P}}}{M_{\Omega}^{P}f\left(  x\right)  }\right)  ^{\frac
{q-p}{\alpha q}}$ for every $x\in%
\mathbb{R}
^{n}$, we have%
\[
\left \vert T_{\Omega,\alpha}^{P}f\left(  x\right)  \right \vert \lesssim \left(
M_{\Omega}^{P}f\left(  x\right)  \right)  ^{\frac{p}{q}}\left \Vert
f\right \Vert _{M_{p,\varphi^{\frac{1}{p}},P}}^{1-\frac{p}{q}}.
\]

Consequently the statement of the theorem follows in view of the boundedness
of the maximal operator with rough kernel $M_{\Omega}^{P}$ in $M_{p,\varphi
^{\frac{1}{p}},P}$ provided by Theorem 4.2 in \cite{GUL-BAL} in virtue of
condition (\ref{31*}).

Therefore, we obtain%
\begin{align*}
\left \Vert T_{\Omega,\alpha}^{P}f\right \Vert _{M_{q,\varphi^{\frac{1}{q}},P}}
& =\sup_{x\in%
\mathbb{R}
^{n},\text{ }t>0}\varphi \left(  x,t\right)  ^{-\frac{1}{q}}t^{-\frac{\gamma
}{q}}\left \Vert T_{\Omega,\alpha}^{P}f\right \Vert _{L_{q}\left(  E\left(
x,t\right)  \right)  }\\
& \lesssim \left \Vert f\right \Vert _{M_{p,\varphi^{\frac{1}{p}},P}}^{1-\frac
{p}{q}}\sup_{x\in%
\mathbb{R}
^{n},\text{ }t>0}\varphi \left(  x,t\right)  ^{-\frac{1}{q}}t^{-\frac{\gamma
}{q}}\left \Vert M_{\Omega}^{P}f\right \Vert _{L_{p}\left(  E\left(  x,t\right)
\right)  }^{\frac{p}{q}}\\
& =\left \Vert f\right \Vert _{M_{p,\varphi^{\frac{1}{p}},P}}^{1-\frac{p}{q}%
}\left(  \sup_{x\in%
\mathbb{R}
^{n},\text{ }t>0}\varphi \left(  x,t\right)  ^{-\frac{1}{p}}t^{-\frac{\gamma
}{p}}\left \Vert M_{\Omega}^{P}f\right \Vert _{L_{p}\left(  E\left(  x,t\right)
\right)  }\right)  ^{\frac{p}{q}}\\
& =\left \Vert f\right \Vert _{M_{p,\varphi^{\frac{1}{p}},P}}^{1-\frac{p}{q}%
}\left \Vert M_{\Omega}^{P}f\right \Vert _{M_{p,\varphi^{\frac{1}{p}},P}}%
^{\frac{p}{q}}\\
& \lesssim \left \Vert f\right \Vert _{M_{p,\varphi^{\frac{1}{p}},P}},
\end{align*}
if $1<p<q<\infty$ and%
\begin{align*}
\left \Vert T_{\Omega,\alpha}^{P}f\right \Vert _{M_{q,\varphi^{\frac{1}{q}},P}}
& =\sup_{x\in%
\mathbb{R}
^{n},\text{ }t>0}\varphi \left(  x,t\right)  ^{-\frac{1}{q}}t^{-\frac{\gamma
}{q}}\left \Vert T_{\Omega,\alpha}^{P}f\right \Vert _{WL_{q}\left(  E\left(
x,t\right)  \right)  }\\
& \lesssim \left \Vert f\right \Vert _{M_{1,\varphi,P}}^{1-\frac{1}{q}}\sup_{x\in%
\mathbb{R}
^{n},\text{ }t>0}\varphi \left(  x,t\right)  ^{-\frac{1}{q}}t^{-\frac{\gamma
}{q}}\left \Vert M_{\Omega}^{P}f\right \Vert _{WL_{1}\left(  E\left(
x,t\right)  \right)  }^{\frac{1}{q}}\\
& =\left \Vert f\right \Vert _{M_{1,\varphi,P}}^{1-\frac{1}{q}}\left(
\sup_{x\in%
\mathbb{R}
^{n},\text{ }t>0}\varphi \left(  x,t\right)  ^{-1}t^{-n}\left \Vert M_{\Omega
}^{P}f\right \Vert _{WL_{1}\left(  E\left(  x,t\right)  \right)  }\right)
^{\frac{1}{q}}\\
& =\left \Vert f\right \Vert _{M_{1,\varphi,P}}^{1-\frac{1}{q}}\left \Vert
M_{\Omega}^{P}f\right \Vert _{WM_{1,\varphi,P}}^{\frac{1}{q}}\\
& \lesssim \left \Vert f\right \Vert _{M_{1,\varphi,P}},
\end{align*}
if $1<q<\infty$.

Hence, the proof is completed.
\end{proof}

Before giving the proof of Theorem \ref{teo100}, we introduce some lemmas and
theorems about the estimates of the parabolic sublinear commutator of the
parabolic fractional maximal operator with rough kernel on the parabolic
generalized Morrey spaces.

\begin{lemma}
\label{Lemma 1}Let $\Omega \in L_{s}(S^{n-1})$, $1<s\leq \infty$, be $A_{t}%
$-homogeneous of degree zero. Let $1<p<\infty$, $0<\alpha<\frac{\gamma}{p}$,
$\frac{1}{q}=\frac{1}{p}-\frac{\alpha}{\gamma}$, $b\in BMO_{P}\left(
{\mathbb{R}^{n}}\right)  $ and $M_{\Omega,b,\alpha}^{P}$ is bounded from
$L_{p}({\mathbb{R}^{n}})$ to $L_{q}({\mathbb{R}^{n}})$. Then for $s^{\prime
}<p$, the inequality%
\[
\Vert M_{\Omega,b,\alpha}^{P}f\Vert_{L_{q}(E(x_{0},r))}\lesssim \Vert
b\Vert_{\ast}\,r^{\frac{\gamma}{q}}\sup_{t>2kr}\left(  1+\ln \frac{t}%
{r}\right)  t^{-\frac{\gamma}{q}}\Vert f\Vert_{L_{p}(E(x_{0},t))}%
\]
holds for any ellipsoid $E(x_{0},r)$ and for all $f\in L_{p}^{loc}%
({\mathbb{R}^{n}})$.
\end{lemma}

\begin{proof}
Let $1<p<\infty$, $0<\alpha<\frac{\gamma}{p}$ and $\frac{1}{q}=\frac{1}%
{p}-\frac{\alpha}{\gamma}$. As in the proof of Theorem \ref{teo6}, we
represent $f$ in form (\ref{e39}) and have%
\[
\left \Vert M_{\Omega,b,\alpha}^{P}f\right \Vert _{L_{q}\left(  E\right)  }%
\leq \left \Vert M_{\Omega,b,\alpha}^{P}f_{1}\right \Vert _{L_{q}\left(
E\right)  }+\left \Vert M_{\Omega,b,\alpha}^{P}f_{2}\right \Vert _{L_{q}\left(
E\right)  }.
\]
From the boundedness of $M_{\Omega,b,\alpha}^{P}$ from $L_{p}({\mathbb{R}^{n}%
})$ to $L_{q}({\mathbb{R}^{n}})$ (see Corollary 0.1 in \cite{Gurbuz2}) it
follows that:%
\begin{align*}
\left \Vert M_{\Omega,b,\alpha}^{P}f_{1}\right \Vert _{L_{q}\left(  E\right)  }
& \leq \left \Vert M_{\Omega,b,\alpha}^{P}f_{1}\right \Vert _{L_{q}\left(
{\mathbb{R}^{n}}\right)  }\\
& \lesssim \left \Vert b\right \Vert _{\ast}\left \Vert f_{1}\right \Vert
_{L_{p}\left(  {\mathbb{R}^{n}}\right)  }=\left \Vert b\right \Vert _{\ast
}\left \Vert f\right \Vert _{L_{p}\left(  2kE\right)  }.
\end{align*}
For $x\in E$, we have%
\begin{align*}
M_{\Omega,b,\alpha}^{P}f_{2}\left(  x\right)   & \lesssim \sup_{t>0}\frac
{1}{\left \vert E(x,t)\right \vert ^{1-\frac{\alpha}{\gamma}}}\int
\limits_{E\left(  x,t\right)  }\left \vert \Omega \left(  x-y\right)
\right \vert \left \vert b\left(  y\right)  -b\left(  x\right)  \right \vert
\left \vert f_{2}\left(  y\right)  \right \vert dy\\
& =\sup_{t>0}\frac{1}{\left \vert E(x,t)\right \vert ^{1-\frac{\alpha}{\gamma}}%
}\int \limits_{E\left(  x,t\right)  \cap \left(  2kE\right)  ^{C}}\left \vert
\Omega \left(  x-y\right)  \right \vert \left \vert b\left(  y\right)  -b\left(
x\right)  \right \vert \left \vert f\left(  y\right)  \right \vert dy.
\end{align*}

Let $x$ be an arbitrary point from $E$. If $E\left(  x,t\right)  \cap \left \{
\left(  2kE\right)  ^{C}\right \}  \neq \emptyset$, then $t>r$. Indeed, if $y\in
E\left(  x,t\right)  \cap \left \{  \left(  2kE\right)  ^{C}\right \}  $, then
$t>\rho \left(  x-y\right)  \geq \frac{1}{k}\rho \left(  x_{0}-y\right)
-\rho \left(  x_{0}-x\right)  >2r-r=r$.

On the other hand, $E\left(  x,t\right)  \cap \left \{  \left(  2kE\right)
^{C}\right \}  \subset E\left(  x_{0},2kt\right)  $. Indeed, $y\in E\left(
x,t\right)  \cap \left \{  \left(  2kE\right)  ^{C}\right \}  $, then we get
$\rho \left(  x_{0}-y\right)  \leq k\rho \left(  x-y\right)  +k\rho \left(
x_{0}-x\right)  <k\left(  t+r\right)  <2kt$.

Hence%
\begin{align*}
M_{\Omega,b,\alpha}^{P}f_{2}\left(  x\right)   & =\sup_{t>0}\frac
{1}{\left \vert E(x,t)\right \vert ^{1-\frac{\alpha}{n}}}\int \limits_{E\left(
x,t\right)  \cap \left(  2kE\right)  ^{C}}\left \vert \Omega \left(  x-y\right)
\right \vert \left \vert b\left(  y\right)  -b\left(  x\right)  \right \vert
\left \vert f\left(  y\right)  \right \vert dy\\
& \leq \left(  2k\right)  ^{\gamma-\alpha}\sup_{t>r}\frac{1}{\left \vert
E(x_{0},2kt)\right \vert ^{1-\frac{\alpha}{\gamma}}}\int \limits_{E\left(
x_{0},2kt\right)  }\left \vert \Omega \left(  x-y\right)  \right \vert \left \vert
b\left(  y\right)  -b\left(  x\right)  \right \vert \left \vert f\left(
y\right)  \right \vert dy\\
& =\left(  2k\right)  ^{\gamma-\alpha}\sup_{t>2kr}\frac{1}{\left \vert
E(x_{0},t)\right \vert ^{1-\frac{\alpha}{\gamma}}}\int \limits_{E\left(
x_{0},t\right)  }\left \vert \Omega \left(  x-y\right)  \right \vert \left \vert
b\left(  y\right)  -b\left(  x\right)  \right \vert \left \vert f\left(
y\right)  \right \vert dy.
\end{align*}

Therefore, for all $x\in E$ we have%
\begin{equation}
M_{\Omega,b,\alpha}^{P}f_{2}\left(  x\right)  \leq \left(  2k\right)
^{\gamma-\alpha}\sup_{t>2kr}\frac{1}{\left \vert E(x_{0},t)\right \vert
^{1-\frac{\alpha}{\gamma}}}\int \limits_{E\left(  x_{0},t\right)  }\left \vert
\Omega \left(  x-y\right)  \right \vert \left \vert b\left(  y\right)  -b\left(
x\right)  \right \vert \left \vert f\left(  y\right)  \right \vert dy.\label{63}%
\end{equation}

Then%
\begin{align*}
\left \Vert M_{\Omega,b,\alpha}^{P}f_{2}\right \Vert _{L_{q}\left(  E\right)  }
& \lesssim \left(  \int \limits_{E}\left(  \sup_{t>2kr}\frac{1}{\left \vert
E(x_{0},t)\right \vert ^{1-\frac{\alpha}{\gamma}}}\int \limits_{E\left(
x_{0},t\right)  }\left \vert \Omega \left(  x-y\right)  \right \vert \left \vert
b\left(  y\right)  -b\left(  x\right)  \right \vert \left \vert f\left(
y\right)  \right \vert dy\right)  ^{q}dx\right)  ^{\frac{1}{q}}\\
& \leq \left(  \int \limits_{E}\left(  \sup_{t>2kr}\frac{1}{\left \vert
E(x_{0},t)\right \vert ^{1-\frac{\alpha}{\gamma}}}\int \limits_{E\left(
x_{0},t\right)  }\left \vert \Omega \left(  x-y\right)  \right \vert \left \vert
b\left(  y\right)  -b_{E}\right \vert \left \vert f\left(  y\right)  \right \vert
dy\right)  ^{q}dx\right)  ^{\frac{1}{q}}\\
& +\left(  \int \limits_{E}\left(  \sup_{t>2kr}\frac{1}{\left \vert
E(x_{0},t)\right \vert ^{1-\frac{\alpha}{\gamma}}}\int \limits_{E\left(
x_{0},t\right)  }\left \vert \Omega \left(  x-y\right)  \right \vert \left \vert
b\left(  x\right)  -b_{E}\right \vert \left \vert f\left(  y\right)  \right \vert
dy\right)  ^{q}dx\right)  ^{\frac{1}{q}}\\
& =J_{1}+J_{2}.
\end{align*}

Let us estimate $J_{1}$.%
\begin{align*}
J_{1}  & =r^{\frac{\gamma}{q}}\sup_{t>2kr}\frac{1}{\left \vert E(x_{0}%
,t)\right \vert ^{1-\frac{\alpha}{\gamma}}}\int \limits_{E\left(  x_{0}%
,t\right)  }\left \vert \Omega \left(  x-y\right)  \right \vert \left \vert
b\left(  y\right)  -b_{E}\right \vert \left \vert f\left(  y\right)  \right \vert
dy\\
& \approx r^{\frac{\gamma}{q}}\sup_{t>2kr}t^{\alpha-\gamma}\int
\limits_{E\left(  x_{0},t\right)  }\left \vert \Omega \left(  x-y\right)
\right \vert \left \vert b\left(  y\right)  -b_{E}\right \vert \left \vert
f\left(  y\right)  \right \vert dy.
\end{align*}

Applying H\"{o}lder's inequality, by $\left[  \text{(2.8) in \cite{Gurbuz2}%
}\right]  $, (\ref{5.1}), (\ref{5.2}) and $\frac{1}{\mu}+\frac{1}{p}+\frac
{1}{s}=1$ we get%
\begin{align*}
J_{1}  & \lesssim r^{\frac{\gamma}{q}}\sup_{t>2kr}t^{\alpha-\gamma}%
\int \limits_{E\left(  x_{0},t\right)  }\left \vert \Omega \left(  x-y\right)
\right \vert \left \vert b\left(  y\right)  -b_{E\left(  x_{0},t\right)
}\right \vert \left \vert f\left(  y\right)  \right \vert dy\\
& +r^{\frac{\gamma}{q}}\sup_{t>2kr}t^{\alpha-\gamma}\left \vert b_{E\left(
x_{0},r\right)  }-b_{E\left(  x_{0},t\right)  }\right \vert \int
\limits_{E\left(  x_{0},t\right)  }\left \vert \Omega \left(  x-y\right)
\right \vert \left \vert f\left(  y\right)  \right \vert dy\\
& \lesssim r^{\frac{\gamma}{q}}\sup_{t>2kr}t^{\alpha-\frac{\gamma}{p}%
}\left \Vert \Omega \left(  \cdot-y\right)  \right \Vert _{L_{s}\left(  E\left(
x_{0},t\right)  \right)  }\left \Vert \left(  b\left(  \cdot \right)
-b_{E\left(  x_{0},t\right)  }\right)  \right \Vert _{L_{\mu}\left(  E\left(
x_{0},t\right)  \right)  }\left \Vert f\right \Vert _{L_{p}\left(  E\left(
x_{0},t\right)  \right)  }\\
& +r^{\frac{\gamma}{q}}\sup_{t>2kr}t^{\alpha-\gamma}\left \vert b_{E\left(
x_{0},r\right)  }-b_{E\left(  x_{0},t\right)  }\right \vert \left \Vert
\Omega \left(  \cdot-y\right)  \right \Vert _{L_{s}\left(  E\left(
x_{0},t\right)  \right)  }\left \Vert f\right \Vert _{L_{p}\left(  E\left(
x_{0},t\right)  \right)  }\left \vert E\left(  x_{0},t\right)  \right \vert
^{1-\frac{1}{p}-\frac{1}{s}}\\
& \lesssim \Vert b\Vert_{\ast}\,r^{\frac{\gamma}{q}}\sup_{t>2kr}\left(
1+\ln \frac{t}{r}\right)  t^{-\frac{\gamma}{q}}\Vert f\Vert_{L_{p}(E(x_{0}%
,t))}.
\end{align*}

In order to estimate $J_{2}$ note that%
\[
J_{2}=\left \Vert \left(  b\left(  \cdot \right)  -b_{E\left(  x_{0},t\right)
}\right)  \right \Vert _{L_{q}\left(  E\left(  x_{0},t\right)  \right)  }%
\sup_{t>2kr}t^{\alpha-\gamma}\int \limits_{E\left(  x_{0},t\right)  }\left \vert
\Omega \left(  x-y\right)  \right \vert \left \vert f\left(  y\right)
\right \vert dy.
\]

By (\ref{5.1}), we get%
\[
J_{2}\lesssim \Vert b\Vert_{\ast}\,r^{\frac{\gamma}{q}}\sup_{t>2kr}%
t^{\alpha-\gamma}\int \limits_{E\left(  x_{0},t\right)  }\left \vert
\Omega \left(  x-y\right)  \right \vert \left \vert f\left(  y\right)
\right \vert dy.
\]

Thus, by (\ref{34}) and $\left[  \text{(2.5) in \cite{Gurbuz2}}\right]  $
\[
J_{2}\lesssim \Vert b\Vert_{\ast}\,r^{\frac{\gamma}{q}}\sup_{t>2kr}%
t^{-\frac{\gamma}{q}}\left \Vert f\right \Vert _{L_{p}\left(  E\left(
x_{0},t\right)  \right)  }.
\]

Summing up $J_{1}$ and $J_{2}$, for all $p\in \left(  1,\infty \right)  $ we get%
\begin{equation}
\left \Vert M_{\Omega,b,\alpha}^{P}f_{2}\right \Vert _{L_{q}\left(  E\right)
}\lesssim \Vert b\Vert_{\ast}\,r^{\frac{\gamma}{q}}\sup_{t>2kr}t^{-\frac
{\gamma}{q}}\left(  1+\ln \frac{t}{r}\right)  \left \Vert f\right \Vert
_{L_{p}\left(  E\left(  x_{0},t\right)  \right)  }.\label{64}%
\end{equation}

Finally, combining $\left \Vert M_{\Omega,b,\alpha}^{P}f_{1}\right \Vert
_{L_{q}\left(  E\right)  }$ and $\left \Vert M_{\Omega,b,\alpha}^{P}%
f_{2}\right \Vert _{L_{q}\left(  E\right)  }$we have the following%
\begin{align*}
\left \Vert M_{\Omega,b,\alpha}^{P}f\right \Vert _{L_{q}\left(  E\right)  }  &
\lesssim \left \Vert b\right \Vert _{\ast}\left \Vert f\right \Vert _{L_{p}\left(
2kE\right)  }+\Vert b\Vert_{\ast}\,r^{\frac{\gamma}{q}}\sup_{t>2kr}%
t^{-\frac{\gamma}{q}}\left(  1+\ln \frac{t}{r}\right)  \left \Vert f\right \Vert
_{L_{p}\left(  E\left(  x_{0},t\right)  \right)  }\\
& \lesssim \Vert b\Vert_{\ast}\,r^{\frac{\gamma}{q}}\sup_{t>2kr}t^{-\frac
{\gamma}{q}}\left(  1+\ln \frac{t}{r}\right)  \left \Vert f\right \Vert
_{L_{p}\left(  E\left(  x_{0},t\right)  \right)  },
\end{align*}
which completes the proof.
\end{proof}

Similar to Lemma \ref{Lemma 1} the following lemma can also be proved.

\begin{lemma}
Let $\Omega \in L_{s}(S^{n-1})$, $1<s\leq \infty$, be $A_{t}$-homogeneous of
degree zero. Let $1<p<\infty$, $b\in BMO_{P}\left(  {\mathbb{R}^{n}}\right)  $
and $M_{\Omega,b}^{P}$ is bounded on $L_{p}({\mathbb{R}^{n}})$. Then for
$s^{\prime}<p$, the inequality%
\[
\Vert M_{\Omega,b}^{P}f\Vert_{L_{p}(E(x_{0},r))}\lesssim \Vert b\Vert_{\ast
}\,r^{\frac{\gamma}{q}}\sup_{t>2kr}\left(  1+\ln \frac{t}{r}\right)
t^{-\frac{\gamma}{p}}\Vert f\Vert_{L_{p}(E(x_{0},t))}%
\]
holds for any ellipsoid $E(x_{0},r)$ and for all $f\in L_{p}^{loc}%
({\mathbb{R}^{n}})$.
\end{lemma}

The following theorem is true.

\begin{theorem}
\label{teo10}Let $\Omega \in L_{s}(S^{n-1})$, $1<s\leq \infty$, be $A_{t}%
$-homogeneous of degree zero. Let $1<p<\infty$, $0\leq \alpha<\frac{\gamma}{p}%
$, $\frac{1}{q}=\frac{1}{p}-\frac{\alpha}{\gamma}$, $b\in BMO_{P}\left(
{\mathbb{R}^{n}}\right)  $ and let $\left(  \varphi_{1},\varphi_{2}\right)  $
satisfies the condition%
\[
\sup_{r<t<\infty}t^{\alpha-\frac{\gamma}{p}}\left(  1+\ln \frac{t}{r}\right)
\operatorname*{essinf}\limits_{t<\tau<\infty}\varphi_{1}\left(  x,\tau \right)
t^{\frac{\gamma}{p}}\leq C\varphi_{2}\left(  x,r\right)  ,
\]
where $C$ does not depend on $x$ and $r$. Then for $s^{\prime}<p$, the
operator $M_{\Omega,b,\alpha}^{P}$ is bounded from $M_{p,\varphi_{1},P}$ to
$M_{q,\varphi_{2},P}$. Moreover%
\[
\left \Vert M_{\Omega,b,\alpha}^{P}f\right \Vert _{M_{q,\varphi_{2},P}}%
\lesssim \Vert b\Vert_{\ast}\left \Vert f\right \Vert _{M_{p,\varphi_{1},P}}.
\]

\end{theorem}

\begin{proof}
The statement of Theorem \ref{teo10} follows by Lemma \ref{Lemma 1} in the
same manner as in the proof of Theorem 4.1 in \cite{GUL-BAL}.
\end{proof}

In the case of $\alpha=0$ and $p=q$, we get the following corollary by Theorem
\ref{teo10}.

\begin{corollary}
\label{corollary2}Let $\Omega \in L_{s}(S^{n-1})$, $1<s\leq \infty$, be $A_{t}%
$-homogeneous of degree zero. Let $1<p<\infty$, $b\in BMO_{P}\left(
{\mathbb{R}^{n}}\right)  $ and let $\left(  \varphi_{1},\varphi_{2}\right)  $
satisfies the condition%
\[
\sup_{r<t<\infty}t^{-\frac{\gamma}{p}}\left(  1+\ln \frac{t}{r}\right)
\operatorname*{essinf}\limits_{t<\tau<\infty}\varphi_{1}\left(  x,\tau \right)
t^{\frac{\gamma}{p}}\leq C\varphi_{2}\left(  x,r\right)  ,
\]
where $C$ does not depend on $x$ and $r$. Then for $s^{\prime}<p$, the
operator $M_{\Omega,b}^{P}$ is bounded from $M_{p,\varphi_{1},P}$ to
$M_{q,\varphi_{2},P}$. Moreover%
\[
\left \Vert M_{\Omega,b}^{P}f\right \Vert _{M_{q,\varphi_{2},P}}\lesssim \Vert
b\Vert_{\ast}\left \Vert f\right \Vert _{M_{p,\varphi_{1},P}}.
\]

\end{corollary}

Now we are ready to return to the proof of Theorem \ref{teo100}.

\subsection{Proof of Theorem \ref{teo100}}

\begin{proof}
Let $1<p<\infty$, $0<\alpha<\frac{n}{p}$ and $\frac{1}{q}=\frac{1}{p}%
-\frac{\alpha}{n}$, $p<q$ and $f\in M_{p,\varphi^{\frac{1}{p}},P}$. As in the
proof of Theorem \ref{teo6}, we represent $f$ in form (\ref{e39}) and have%
\[
\left \Vert \lbrack b,T_{\Omega,\alpha}^{P}]f\right \Vert _{L_{q}\left(
E\right)  }\leq \left \Vert \lbrack b,T_{\Omega,\alpha}^{P}]f_{1}\right \Vert
_{L_{q}\left(  E\right)  }+\left \Vert [b,T_{\Omega,\alpha}^{P}]f_{2}%
\right \Vert _{L_{q}\left(  E\right)  }.
\]

For $x\in E$ we have%
\[
\left \vert \lbrack b,T_{\Omega,\alpha}^{P}]f_{2}\left(  x\right)  \right \vert
\lesssim \int \limits_{\left(  2kE\right)  ^{C}}\frac{\left \vert \Omega \left(
x-y\right)  \right \vert }{\rho \left(  x-y\right)  ^{\gamma-\alpha}}\left \vert
b\left(  y\right)  -b\left(  x\right)  \right \vert \left \vert f\left(
y\right)  \right \vert dy.
\]

Analogously to Section 3.1, for all $p\in \left(  1,\infty \right)  $ and $x\in
E$ we get%
\begin{equation}
\left \vert \lbrack b,T_{\Omega,\alpha}^{P}]f_{2}\left(  x\right)  \right \vert
\lesssim \Vert b\Vert_{\ast}\, \int \limits_{2kr}^{\infty}\left(  1+\ln \frac
{t}{r}\right)  t^{\alpha-\frac{\gamma}{p}-1}\left \Vert f\right \Vert
_{L_{p}\left(  E\left(  x,t\right)  \right)  }dt.\label{75}%
\end{equation}

Then from conditions (\ref{74}), (\ref{100}) and inequality (\ref{75}) we get%
\begin{align}
\left \vert \lbrack b,T_{\Omega,\alpha}^{P}]f\left(  x\right)  \right \vert  &
\lesssim \Vert b\Vert_{\ast}r^{\alpha}M_{\Omega,b}^{P}f\left(  x\right)  +\Vert
b\Vert_{\ast}\int \limits_{2kr}^{\infty}\left(  1+\ln \frac{t}{r}\right)
t^{\alpha-\frac{\gamma}{p}-1}\left \Vert f\right \Vert _{L_{p}\left(  E\left(
x,t\right)  \right)  }dt\nonumber \\
& \leq \Vert b\Vert_{\ast}r^{\alpha}M_{\Omega,b}^{P}f\left(  x\right)  +\Vert
b\Vert_{\ast}\left \Vert f\right \Vert _{M_{p,\varphi^{\frac{1}{p}},P}}%
\int \limits_{2kr}^{\infty}\left(  1+\ln \frac{t}{r}\right)  t^{\alpha}%
\varphi \left(  x,t\right)  ^{\frac{1}{p}}\frac{dt}{t}\nonumber \\
& \lesssim \Vert b\Vert_{\ast}r^{\alpha}M_{\Omega,b}^{P}f\left(  x\right)
+\Vert b\Vert_{\ast}r^{-\frac{\alpha p}{q-p}}\left \Vert f\right \Vert
_{M_{p,\varphi^{\frac{1}{p}},P}}.\label{76}%
\end{align}

Hence choosing $r=\left(  \frac{\left \Vert f\right \Vert _{M_{p,\varphi
^{\frac{1}{p}},P}}}{M_{\Omega,b}^{P}f\left(  x\right)  }\right)  ^{\frac
{q-p}{\alpha q}}$ for every $x\in%
\mathbb{R}
^{n}$, we have%
\[
\left \vert \lbrack b,T_{\Omega,\alpha}^{P}]f\left(  x\right)  \right \vert
\lesssim \Vert b\Vert_{\ast}\left(  M_{\Omega,b}^{P}f\left(  x\right)  \right)
^{\frac{p}{q}}\left \Vert f\right \Vert _{M_{p,\varphi^{\frac{1}{p}},P}%
}^{1-\frac{p}{q}}.
\]

Consequently the statement of the theorem follows in view of the boundedness
of the commutator of the parabolic maximal operator with rough kernel
$M_{\Omega,b}^{P}$ in $M_{p,\varphi^{\frac{1}{p}},P}$ provided by Corollary
\ref{corollary2} in virtue of condition (\ref{67}).

Therefore, we have%
\begin{align*}
\left \Vert \lbrack b,T_{\Omega,\alpha}^{P}]f\right \Vert _{M_{q,\varphi
^{\frac{1}{q}},P}}  & =\sup_{x\in%
\mathbb{R}
^{n},\text{ }t>0}\varphi \left(  x,t\right)  ^{-\frac{1}{q}}t^{-\frac{\gamma
}{q}}\left \Vert [b,T_{\Omega,\alpha}^{P}]f\right \Vert _{L_{q}\left(  E\left(
x,t\right)  \right)  }\\
& \lesssim \Vert b\Vert_{\ast}\left \Vert f\right \Vert _{M_{p,\varphi^{\frac
{1}{p}},P}}^{1-\frac{p}{q}}\sup_{x\in%
\mathbb{R}
^{n},\text{ }t>0}\varphi \left(  x,t\right)  ^{-\frac{1}{q}}t^{-\frac{\gamma
}{q}}\left \Vert M_{\Omega,b}^{P}f\right \Vert _{L_{p}\left(  E\left(
x,t\right)  \right)  }^{\frac{p}{q}}\\
& =\Vert b\Vert_{\ast}\left \Vert f\right \Vert _{M_{p,\varphi^{\frac{1}{p}},P}%
}^{1-\frac{p}{q}}\left(  \sup_{x\in%
\mathbb{R}
^{n},\text{ }t>0}\varphi \left(  x,t\right)  ^{-\frac{1}{p}}t^{-\frac{\gamma
}{p}}\left \Vert M_{\Omega,b}^{P}f\right \Vert _{L_{p}\left(  E\left(
x,t\right)  \right)  }\right)  ^{\frac{p}{q}}\\
& =\Vert b\Vert_{\ast}\left \Vert f\right \Vert _{M_{p,\varphi^{\frac{1}{p}},P}%
}^{1-\frac{p}{q}}\left \Vert M_{\Omega,b}^{P}f\right \Vert _{M_{p,\varphi
^{\frac{1}{p}}}}^{\frac{p}{q}}\\
& \lesssim \Vert b\Vert_{\ast}\left \Vert f\right \Vert _{M_{p,\varphi^{\frac
{1}{p}},P}}.
\end{align*}

\end{proof}

\end{document}